\documentclass[12pt]{amsart}
\usepackage{amssymb,latexsym,amsmath,amscd,mathrsfs,yfonts,hyperref}
\usepackage{fullpage}
\input xy
\xyoption{all}

\newcommand{\RR}{\mathbb R}
\newcommand{\ZZ}{\mathbb Z}
\newcommand{\QQ}{\mathbb Q}
\newcommand{\CC}{\mathbb C}

\newcommand{\NN}{\mathbb N}

\newcommand{\MM}{\mathbb M}

\newcommand{\cpt}{\mathbb K}

\def\bK{\mathbf K}

\def\Aut{\textup{Aut}}

\def\Coh{\mathtt{Coh}}

\def\Id{\textup{Id}}

\def\SL{\textup{SL}}
\def\Int{\textup{Int}}
\def\per{\textup{per}}

\def\Perf{\mathtt{Perf}}

\def\Cone{\textup{Cone}}
\def\C{\textup{C}}

\def\KKcat{\mathtt{KK_{C^*}}}
\def\Z{\textup{Z}}
\def\KQ{\textup{KQ}}

\def\id{\mathrm{id}}
\def\Im{\textup{Im}}
\def\IAlg{\mathtt{Alg}^{\mathtt{Idem}}_k}

\def\Hom{\textup{Hom}}

\def\kAlg{\mathtt{Alg}_k}

\def\ker{\textup{ker}}

\def\op{\textup{op}}

\def\rep{\textup{rep}}

\def\Hqe{\mathtt{Hqe}}
\def\Hmo{\mathtt{Hmo}}

\def\K{\textup{K}}
\def\SS{\mathbb{S}}
\def\alg{\textup{alg}}

\def\prot{\hat{\otimes}}

\def\bHC{\mathbf{HC}}
\def\cJ{\mathcal J}

\def\NCS{\mathtt{NCS_{dg}}}
\def\DGcorr{\mathtt{NCC_{dg}}}

\def\CSp{\mathtt{NCS_{C^*}}}

\def\HoSpt{\mathtt{HoSpt}}

\def\DGcat{\mathtt{DGcat}}
\def\CAlg{\mathtt{Alg_{C^*}}}

\def\KK{\textup{KK}}
\def\E{\textup{E}}

\def\H{\textup{H}}

\def\cC{\mathcal C}

\def\1{\bf{1}}

\def\hofib{\textup{hofib}}
\def\Csep{\mathtt{Sep_{C^*}}}

\def\kQ{\mathbf{KQ}}
\def\Pr{{\sf HPf_{dg}}}
\def\PrK{{\sf HPf_{dg}^{\textup{st}}}}

\def\cD{\mathcal D}

\def\cM{\mathcal M}

\def\TT{\mathbb{T}}

\def\A{\mathcal{A}}
\def\cA{\mathcal{A}}
\def\cB{\mathcal{B}}

\newcommand{\map}{\rightarrow}
\newcommand{\functor}{\longrightarrow}

\newcommand{\beq}{\begin{eqnarray}}
\newcommand{\beqn}{\begin{eqnarray*}}
\newcommand{\eeq}{\end{eqnarray}}
\newcommand{\eeqn}{\end{eqnarray*}}

\newtheorem{thm}{Theorem}[section]

\newtheorem{lem}[thm]{Lemma}
\newtheorem{prop}[thm]{Proposition}
\newtheorem{cor}[thm]{Corollary}
\newtheorem{ex}[thm]{Example}
\newtheorem{defn}[thm]{Definition}
\newtheorem{rem}[thm]{Remark}

\begin{document}

\title{Higher nonunital Quillen $\K'$-theory, $\KK$-dualities and applications to topological $\TT$-dualities}
\author{Snigdhayan Mahanta}

\email{snigdhayan.mahanta@mathematik.uni-regensburg.de}
\address{Fakult{\"a}t f{\"u}r Mathematik, Universit{\"a}t Regensburg, 93040 Regensburg, Germany.}

\subjclass[2010]{19Dxx, 46L85, 58B34, 57R56}
\keywords{$C^*$-algebra, (bivariant) $\K$-theory, differential graded category, $\mathbb{T}$-duality}
\maketitle

\begin{abstract}
Quillen introduced a new $\K'_0$-theory of nonunital rings in \cite{QuiNonunitalK0} and showed that, under some assumptions (weaker than the existence of unity), this new theory agrees with the usual algebraic $\K^\alg_0$-theory. For a field $k$ of characteristic $0$, we introduce higher nonunital $\K$-theory of $k$-algebras, denoted $\KQ$, which extends Quillen's original definition of the $\K'_0$ functor. We show that the $\KQ$-theory is Morita invariant and satisfies excision connectively, in a suitable sense, on the category of idempotent $k$-algebras. Using these two properties we show that the $\KQ$-theory agrees with the topological $\K$-theory of stable $C^*$-algebras. The machinery enables us to produce a DG categorical formalism of topological homological $\TT$-duality using bivariant $\K$-theory classes. A connection with strong deformations of $C^*$-algebras and some other potential applications to topological field theories are discussed towards the end.
\end{abstract}

\begin{center}
{\bf Introduction}
\end{center}

Quillen defined nonunital $\K'_0$-groups for possibly nonunital rings in \cite{QuiNonunitalK0} and showed that a rather general type of Morita context induces an isomorphism between these groups. Using some machinery involving homotopy theory of differential graded (DG) categories developed by Keller--Tabuada, we introduce higher analogues of Quillen's $\K'_0$-groups for $k$-algebras, where $k$ is a field of characteristic $0$, and call it the $\KQ$-theory. It is desirable to extend the theory to algebras over an arbitrary commutative and unital ring. We establish some properties of the $\KQ$-theory which make them amenable to calculation, e.g., Morita invariance and excision (under some restrictions). In the following two paragraphs we explain what motivated us to study the $\KQ$-theory. The $\KQ$-groups are defined to be the $\K$-theory groups of certain DG categories. It will turn out that our primary interest lies in these DG categories, rather than their $\K$-theory groups.

In the operator algebraic setting some connections between Kasparov's $\KK$-theory and noncommutative $\TT$-dualities were investigated in \cite{BMRS,BMRS2}. Let $\KKcat$ be the category of separable $C^*$-algebras, whose morphisms are the bivariant $\KK_0$-groups. We put forward the idea that certain correspondence-like isomorphisms in $\KKcat$ can be regarded as topological homological $\TT$-duality isomorphisms, which give rise to isomorphisms between `DG categories of topological bundles' (in a suitably linearized homotopy category). For brevity, henceforth, we omit the adjective {\it homological} in $\TT$-duality. In general, it is better to have a `duality isomorphism' at the level of DG categories, as it retains more information about the underlying topology than $\K$-theory or cyclic homology. There are instances when the topological information of a principal torus bundle is completely encapsulated by a {\it continuous trace} stable $C^*$-algebra and then constructing its $\TT$-dual becomes an interesting exercise. We construct a functor $\Pr$ from the category of $C^*$-algebras to a category of noncommutative DG correspondences $\DGcorr$ (with DG categories as objects). The functor sends a unital $C^*$-algebra to its {\it homotopy perfect DG category}, which is closely related to the DG category of bounded complexes of finitely generated and projective right modules. This category could be viewed as a topological precursor of the derived DG category of branes on a manifold. One needs to incorporate a lot more geometric data into a $C^*$-algebra, in order to make the previous statement concrete. The appropriate objects for geometry could be something close to a Connes' spectral triple (see, e.g., the reconstruction Theorem \cite{STRec}) and presumably a DG category, whose homotopy category is {\it geometric} in the sense of Kontsevich \cite{KonNotes,KKP}. 

The following picture emerges from various examples \cite{BMRS,BMRS2}: If two stable $C^*$-algebras arise from `$\TT$-dual geometries', then there is a `$\TT$-duality $\KK_*$-class', which accounts for the $\TT$-duality phenomenon. A $\KK_0$-class (resp. $\KK_1$-class) corresponds to an even (resp. odd) number of $\mathbb{T}$-duality transformations. The category $\KKcat$ is triangulated (see \cite{MeyNes}) and one can suspend one of the algebras to convert the $\KK_1$-class into a $\KK_0$-class, i.e., a morphism in $\KKcat$. By suitably stabilizing the functor $\Pr$, one can obtain a new additive functor $\PrK$, which  factors through $\KKcat$. We propose to consider the effect of the `even $\TT$-duality class' in $\KK_0$ on DG categories in $\DGcorr$ via the functor $\PrK$ as a DG categorical manifestation of topological $\TT^{2n}$-duality. If $A, B$ are stable $C^*$-algebras and the $\KK_0(A,B)$-class is invertible, then the induced morphism between the $\K$-theories of the two DG categories $\PrK(A)$ and $\PrK(B)$ gives rise to an isomorphism between the topological $\K$-theories of $A$ and $B$ (using Theorem \ref{KSpec}). This phenomenon agrees with the well-known interchange of topological $\K^{\textup{top}}_0$ and $\K^{\textup{top}}_1$ under $\TT^{2n+1}$-duality transformations, since one of the algebras was suspended to convert the original $\KK_1$-class into a $\KK_0$-class and topological $\K$-theory is Bott $2$-periodic. The article is organized as follows:

In Section \ref{NCC} we describe the noncommutative DG correspondence category $\DGcorr$ and the noncommutative $C^*$-correspondence category $\KKcat$. The correspondence categories $\DGcorr$ and $\KKcat$ are suitably linearized versions of certain categories of noncommutative spaces $\NCS$ and $\CSp$ respectively. The category $\NCS$ (resp. $\DGcorr$) is a slighly enlarged version of the Morita homotopy category of DG categories (resp. linearized version thereof), whose construction was outlined in \cite{TabThesis} and denoted $\Hmo$ (resp. $\Hmo_0$). The category $\DGcorr$ is a good candidate for the category of noncommutative correspondences in the setting of DG categories. In the operator algebraic setting we propose Kasparov's bivariant $\K$-theory category $\KKcat$ as a candidate for noncommutative correspondences. Most of the interesting homological invariants, that we know on DG categories (resp. $C^*$-algebras), factor through $\DGcorr$ (resp. $\KKcat$). 

In Section \ref{KQ} we introduce higher nonunital $\K$-theory after Quillen (see Definition \ref{NKQ}), denoted $\KQ$, and establish some elementary properties of this theory. Let $\kAlg$ denote the category of $k$-algebras, where $k$ is any field of characteristic $0$. We construct a functor $\Pr:\kAlg\functor\NCS$, whose (connective) Waldhausen $\K$-theory is defined to be the $\kQ$ spectrum of $k$-algebras. It is plausible that the $\KQ$-theory is interesting for a wider class of algebras than treated in this article. At the heart of this construction lies Quillen's description of the $\K'_0$-group of nonunital algebras \cite{QuiNonunitalK0}. Then we establish the matrix stability of this theory (see Proposition \ref{KMorInv}). We also show that any idempotent algebra, i.e., $R$ such that $R^2=R$, is connectively $\kQ$-excisive (see Proposition \ref{AlgExcision}). From the work of Suslin--Wodzicki \cite{SusWod2} we know that a $\QQ$-algebra satisfies excision in algebraic $\K$-theory if and only if it is {\it $\H$-unital}. We restrict our attention to the category of $C^*$-algebras with $*$-homomorphisms, denoted $\CAlg$ ($k=\RR$ or $\CC$), which is a (not full) subcategory of the category of idempotent $k$-algebras $\IAlg$. We construct a natural map of spectra from $\kQ(R)$ to the connective algebraic $\K$-theory spectrum $\bK^{\textup{alg}}(R)$ (under some restrictions on $R$), which induces an isomorphism at the level of $\pi_0$, when the input $R$ is a $C^*$-algebra.

In Section \ref{TopK} we establish a general excision property (connectively) of the $\KQ$-theory (see Theorem \ref{excision}) on the category $\IAlg$. Given any $C^*$-algebra, we construct a comparison map from its $\kQ$ spectrum to the connective cover of its topological $\K$-theory spectrum. However, the author does not have a description of the homotopy fibre of this comparison map. It is known that if a functor $F$ from the category of $C^*$-algebras to any other category is matrix stable, then $F(-\prot\cpt)$ is $C^*$-stable, where $\cpt$ is the $C^*$-algebra of compact operators (see Proposition 3.31 of \cite{CunMeyRos}). Using this stabilization and the excision property of the $\KQ$-theory, we show that connectively it agrees with the topological $\K$-theory via the comparison map, when applied to a stable $C^*$-algebra (see Theorem \ref{KSpec}). In the process it is shown that, if the category of noncommutative spaces $\NCS$ is suitably linearized, i.e., converted to the correspondence category $\DGcorr$, then the functor $\Pr$ becomes split exact with values in $\DGcorr$. Using the stabilized additive functor $\PrK=\Pr(-\prot\cpt):\KKcat\map\DGcorr$ (see Theorem \ref{MainThm}) one can also construct actions of the group $\Aut_{\KKcat}(A)$ on $\PrK(A)$ in $\DGcorr$ and consequently on the connective cover of the topological $\K$-theory spectrum $\bK^{\textup{top}}(A\prot\cpt)$ in $\HoSpt$ (see Example \ref{TorusEx}).

In Section \ref{TFT} we speculate about some features that a noncommutative topological field theory should possess and outline how one could possibly reconcile topological $\TT$-duality in that framework. Ideally noncommutative topological field theories, when eventually constructed, should admit the physical dualities, e.g., $\TT$-duality, $\SS$-duality, etc. as examples of explicit isomorphisms between two such theories. In particular, if a $\TT$-duality phenomenon is explained by a $\KK_1$-class, then one should be able to say how this class gives rise to an isomorphism between two theories. In \cite{BloDan} the authors gave a more geometric treatment of various noncommutative dualities using {\it curved DG (bornological) algebras}. However, their constructions were not `factoring through the bivariant category $\KKcat$'. We provide two instances where such factoring arguments can be useful. A discussion about the connection between $\KK$-dualities and topological $\TT$-dualities is the first one (see Example \ref{T-duality}). One can also deduce that, for any nuclear separable $C^*$-algebra $A$, the DG category $\PrK(A)$ is invariant in $\DGcorr$ under {\it strong deformation} of $A$, if $A$ and its strong deformation are suitably homotopy equivalent (see Example \ref{deformation}). Similar results at the level of $\K$-theory groups already exist in the literature \cite{DadLor} (see also \cite{RosKThQuant} for a survey). 

\vspace{3mm}
\noindent
{\bf Notations and conventions:} All algebras are assumed to be associative, but not necessarily commutative or unital, unless explicitly stated so. We are going to work over a ground field $k$ of characteristic $0$ and, while working with topological algebras, the field $k$ will be tacitly assumed to be $\CC$. Unless otherwise stated, all tensor products are assumed to be over $k$. The topological tensor product between two $C^*$-algebras is denoted by $\prot$. Since, whenever we take a tensor product of two $C^*$-algebras, one of the algebras is actually nuclear, we do not need to worry about the distinction between the maximal and the minimal tensor products.

\vspace{3mm}
\noindent
{\bf Acknowledgements.} Under various stages of development of this article the author has benefitted from email correspondences with B. Keller, Matilde Marcolli, Fernando Muro, J. Rosenberg and B. To{\"e}n. The author is particularly grateful to R. Meyer and the anonymous referees for pointing out several inaccuracies in the earlier drafts of this article. The author also gratefully acknowledges the hospitality of the Fields Institute and Institut des Hautes \'Etudes Scientifiques, where much of this work was carried out.

\section{Noncommutative correspondence categories} \label{NCC}

\subsection{The category of DG categories $\DGcat$} \label{DGCAT}
The basic references for the background material, that we require, about the category of small DG categories are \cite{KelDG} and \cite{TabThesis}. The standard references for the theory of model categories are \cite{HomotopicalAlgebra,Hov,Hir}. A {\it DG category} is a category enriched over the symmetric monoidal category of cochain complexes of $k$-linear spaces, so that the composition of morphisms satisfies a Leibniz rule: $d(f\circ g)=df\circ g + (-1)^{|f|}f\circ dg$. Let $\DGcat$ stand for the category of all small DG categories. The morphisms in this category are {\it DG functors}, i.e., (enriched) functors inducing morphisms of $\Hom$-complexes. We provide one generic construction of DG categories which will be useful for later purposes.

\begin{ex} \label{DGex}
Given any $k$-linear category $\mathcal{M}$ it is possible to construct a DG category $\cC_{dg}(\mathcal{M})$ with cochain complexes $(M^\bullet,d_M)$ over $\mathcal{M}$ as objects. One sets $\Hom(M^\bullet,N^\bullet) = \prod_n \Hom(M^\bullet ,N^\bullet)_n$, where $\Hom(M^\bullet,N^\bullet)_n$ denotes the component of morphisms of degree $n$, i.e., $f:M^\bullet\map N^\bullet[n]$ is a map of graded objects and whose differential is the graded commutator $d_N f - (-1)^{|f|}f d_M$.  It is easily seen that the zeroth cocycle category $\Z^0(\cC_{dg}(\cM))$ reduces to the category of cochain complexes over $\cM$ and the zeroth cohomology category $\H^0(\cC_{dg}(\cM))$ produces the homotopy category of complexes over $\cM$.
\end{ex}

Now we recall the notion of the derived category of a DG category as in \cite{Kel}. Let $\cD$ be a small DG category. A right DG $\cD$-module is by definition a DG functor $M:\cD^{\op}\to
\cC_{dg}(k)$, where $\cC_{dg}(k)$ denotes the DG category of cochain complexes of
$k$-linear spaces. We denote the DG category of right DG modules over $\cD$ by $\cC_{dg}(\cD)$. Every object $X$ of $\cD$ defines canonically a {\it free} or {\it representable} right $\cD$-module $h(X):= \Hom_{\cD}(-,X)$. A morphism of DG modules $f:L\to M$ is by definition a natural transformation of DG functors such that $fX:LX\to MX$ is a morphism of complexes for all $X\in\textup{Obj}(\cD)$. We call such an $f$ a quasi-isomorphism if $fX$ is a quasi-isomorphism for all $X$. The derived category $\textup{D}(\cD)$ of $\cD$ is defined to be the localization of the category $\cC_{dg}(\cD)$ with respect to the class of quasi-isomorphisms. This localization does not create any set-theoretic problems, since $\textup{D}(\cD)$ is actually the homotopy category of the {\it projective} model structure on $\cC_{dg}(\cD)$, whose weak equivalences are the quasi-isomorphisms and whose fibrations are the epimorphisms (see Theorem 3.2 of \cite{KelDG}). Every object is fibrant in this model structure and the cofibrant objects are determined by the usual lifting property. In particular, the free or representable modules are cofibrant. The category $\textup{D}(\cD)$ admits a triangulated category structure as explained in Subsection 3.4 of ibid. Given any DG category $\cC$ the category $\H^0(\cC)$, whose objects are the same as $\cC$ and whose morphisms are the zeroth cohomology groups of those of $\cC$, is called the {\it homotopy} category of $\cC$. The Yoneda functor $X\mapsto h(X)$ induces an embedding $\H^0(\cD)\to \textup{D}(\cD)$. The triangulated subcategory of $\textup{D}(\cD)$ generated by the free DG $\cD$-modules $h(X)$ under translations in both directions, formation of mapping cones and passage to direct factors is called the derived category of right {\it perfect $\cD$-modules} and denoted by $\per(\cD)$. The objects of $\per(\cD)$ are precisely the compact objects of $\textup{D}(\cD)$ (see Corollary 3.7 of \cite{KelDG}).

\begin{rem} \label{repDG}
The perfect derived category is idempotent complete. The full subcategory of $\cC_{dg}(\cD)$, consisting of right perfect $\cD$-modules, is denoted by $\per_{dg}(\cD)$, so that one concludes $\H^0(\per_{dg}(\cD))=\per(\cD)$. 
\end{rem}

\subsection{The category of noncommutative DG correspondences $\DGcorr$} \label{MoritaModel} A DG functor $F:\cC\to \cD$ is called a {\it Morita morphism} if the derived base change functor $\mathbb{L}F_*:\textup{D}(\cC)\to \textup{D}(\cD)$ is an equivalence of triangulated categories. There are some other equivalent formulations of a Morita morphism. Thanks to Tabuada \cite{TabThesis} we know that $\DGcat$ has a {\it cofibrantly generated Quillen model category} structure, where the weak equivalences are the Morita morphisms. The generating cofibrations and acyclic cofibrations are described in {\it ibid.} A fibrant object in this model structure is a DG category $\cC$, such that the Yoneda map $\cC\map\per_{dg}(\cC)$ induces an equivalence of categories $\H^0(\cC)\map\per(\cC)$. A {\it quasi-equivalence} $G:\cC\map\cD$ between two DG categories is a DG functor such that it induces a quasi-isomorphism of $\Hom$-complexes and induces an equivalence of homotopy categories $\H^0(\cC)\map\H^0(\cD)$. This model structure is actually a left Bousfield localization of another cofibrantly generated model structure on $\DGcat$, where the weak equivalences are the quasi-equivalences. The homotopy category of $\DGcat$, with quasi-equivalences as weak equivalences, is denoted by $\Hqe$. It is known that if two DG categories are quasi-equivalent, then they are Morita isomorphic, but the converse need not be true. The Morita homotopy category $\Hmo$ is defined to be the localization of $\DGcat$ with respect to the Morita morphisms, i.e., the homotopy category of the Morita model structure on $\DGcat$. Given any DG category $\A$ one constructs the functorial Morita fibrant replacement as $\cA\mapsto\per_{dg}(\cA)$.  This fibrant replacement functor induces an equivalence between $\Hmo$ and a full subcategory of $\Hqe$. 

Tensoring with a cofibrant DG category is not a left Quillen functor from the Morita model category $\DGcat$ to itself as it does not preserve cofibrations. Consequently, there is no adjoint right Quillen inner Hom functor from $\DGcat$ to itself. However, the homotopy category $\Hqe$ has an inner Hom functor which is denoted by $\rep_{dg}(-,?)$ (see \cite{ToeDG}). For the benefit of the reader we recall briefly its construction. Given any DG category $\cA$  one can construct a $\cC_{dg}(k)$-enriched model category structure on the DG category of right DG $\cA$-modules $\textup{D}_{dg}(\cA)$, whose homotopy category turns out to be equivalent to the derived category of $\cA$ \cite{ToeDG}. A DG functor $\phi:\cC\map\cD$ naturally gives rise to a $\cD^{\op}\otimes\cC$-module $M_\phi$, i.e., $M_\phi(-\otimes ?)=\Hom_\cD(\phi(?),-)$. One advantage of working in the DG setting is that every morphism (not necessarily isomorphisms) in $\Hmo$ becomes a generalized DG bimodule morphism or a noncommutative correspondence. In the geometric triangulated setting, e.g., when the triangulated category is of the form $\textup{D}^b(\Coh(X))$ for some smooth and proper variety $X$, such a result is true only for exact equivalences \cite{Orl1}. Let $\Int(\cA)$ denote the category of cofibrant-fibrant objects of this model category, which may be regarded as a $\cC_{dg}(k)$-enrichment of the derived category of $\cA$. If $\cC$ and $\cD$ are DG categories then their inner Hom DG category $\rep_{dg}(\cC,\cD)$ is by definition $\Int(\textup{D}_{dg}(\cD^{\op}\otimes\cC))$, so that $\pi_0(\rep_{dg}(\cC,\cD))\cong\Hom_{\Hqe}(\cC,\cD)$. 

\begin{rem} \label{functoriality}
We conclude that if $F,G:\cC\map\cD$ are DG functors, such that there is a natural isomorphism $F\cong G$, then $F=G$ in $\Hom_{\Hqe}(\cC,\cD)$.
\end{rem}

One obtains an additive category, which is denoted by $\Hmo_0$ in \cite{TabThesis,KelDG}, by applying the functor $\K_0$ (Grothendieck group) to the homotopy triangulated categories of the inner Hom DG categories. The composition of morphisms in $\Hmo_0$ is induced by the derived tensor product of DG categories. The canonical functor $\DGcat\functor\Hmo_0$, which is identity on objects and sends any morphism in $\Hom_{\DGcat}(\cC,\cD)$ to its class in $\K_0(\rep_{dg}(\cC,\cD))$ factors through $\Hmo$ and it is the universal derived Morita invariant and split exact functor \cite{TabThesis}. 

Unfortunately, for many practical purposes (including the applications that we have in mind) small categories are too restrictive. Therefore, we allow the category of {\it noncommutative DG spaces} $\NCS$ and that of {\it noncommutative DG correspondences} $\DGcorr$ to include all essentially (or skeletally) small DG categories. More specifically, we define $\NCS$ [resp. $\DGcorr$] to be the category of all essentially small DG categories, with $\Hom_{\NCS}(\cC,\cD):=\pi_0(\Int(\textup{D}_{dg}(\cD^{\op}\otimes\cC)))$ [resp. $\Hom_{\DGcorr}(\cC,\cD)=\K_0(\Int(\textup{D}_{dg}(\cD^{\op}\otimes\cC)))$]. As a category $\NCS$ [resp. $\DGcorr$] is of course equivalent to $\Hmo$ [resp. $\Hmo_0$]. By abuse of notation, we continue to denote the category of all essentially small DG categories, with DG functors as morphisms, by $\DGcat$.

\begin{rem}
There are certain set-theoretic issues in defining the Morita localization of $\DGcat$, if the objects are no longer assumed to be small. By fixing a large enough universe, these problems can be avoided (see Section 2 of \cite{ToeDG}). Similarly, while defining the $\K$-theory of DG categories (as needed in Section \ref{KQ}), one needs to produce a small Waldhausen category from DG categories and such issues have been addressed in \cite{ThoTro,DugShi}. 
\end{rem}

\subsection{Noncommutative $C^*$-correspondence category $\KKcat$} \label{Corr}
Let us denote the category of all $C^*$-algebras with $*$-homomorphisms as morphisms by $\CAlg$ and the full subcategory consisting of separable $C^*$-algebras by $\Csep$. Typically, while talking about Kasparov's bivariant $\K$-theory we shall restrict our attention to $\Csep$. 

The category of commutative separable $C^*$-algebras corresponds to that of second countable, locally compact and Hausdorff topological spaces. Kasparov developed $\KK$-theory by unifying $\K$-theory and $\K$-homology into a bivariant  $\ZZ/2\ZZ$-graded theory and obtained interesting positive instances of the Baum--Connes conjecture \cite{KasKK1,KasKK2}. A remarkable feature of this theory is the existence of an associative Kasparov product on the $\KK_*$-groups. One can construct a category of separable $C^*$-algebras, where the morphisms are declared to be Kasparov's bivariant $\KK_0$-groups. The Kasparov product is used to define the composition of morphisms. This category, denoted $\KKcat$, plays the role of the category of noncommutative correspondences in the operator algebraic setting (see, for instance, \cite{ConSka,ConConMar1}). This categorical point of view of $\KK$-theory was proposed in \cite{Hig1}. Morphisms in the $\KK_0$-groups can be expressed as homotopy classes of even Kasparov bimodules. Somewhat miraculously in the end one finds that all the analysis disappears and the morphisms are purely determined by some topological data using the framework of quasi-homomorphisms \cite{CunKK}.

Let $\cpt$ denote the algebra of compact operators on a separable Hilbert space. We set $A_\cpt=A\prot\cpt$. A $C^*$-algebra $A$ is called {\it stable} if $A\cong A\prot\cpt$, i.e., $*$-isomorphic as $C^*$-algebras. For instance, the algebra of compact operators $\cpt$ is itself stable. The mapping $A\map A\prot\cpt$ sending $a\longmapsto a\prot\pi$, where $\pi$ is any rank one projection, is called the {\it corner embedding}. This map is clearly nonunital. A functor from $\CAlg$ is called {\it $C^*$-stable} if the corner embedding is mapped to an isomorphism by the functor. There is a canonical functor $\iota:\Csep\functor\KKcat$, where $\Csep$ is the category of separable $C^*$-algebras with $*$-homomorphisms. The functor $\iota$ is identity on the objects and sends a $*$-homomorphism to a $\KK_0$-bimodule class by converting the target into one in the obvious manner. 

\begin{rem} \label{Morita}
There is a counterpart of $\NCS$ in the world of separable $C^*$-algebras, which we denote by $\CSp$. For the details we refer the readers to \cite{Mey1}, where it was called the category of {\it correspondences} in the operator algebraic setting. We regard this category as a category of noncommutative spaces where stably isomorphic algebras are identified. For separable $C^*$-algebras being stably isomorphic is equivalent to being Morita--Rieffel equivalent \cite{BGR}. The objects of $\CSp$ are separable $C^*$-algebras and a morphism $A\map B$ in it is an isomorphism class of a right Hilbert $B_\cpt$-module $\mathcal{E}$ with a nondegenerate $*$-homomorphism $f:A_\cpt \map \cpt(\mathcal{E})$. There is a canonical functor $\Csep\functor\CSp$ which is the universal $C^*$-stable functor on $\Csep$ (Proposition 39 of \cite{Mey1}). Intuitively, the category $\CSp$ is the category of separable $C^*$-algebras with some generalized morphisms, with built-in Morita--Rieffel equivalence. In $\CSp$ any separable $C^*$-algebra $A$ is isomorphic to a stable $C^*$-algebra functorially, {\it viz.,} its own stabilization $A\map A_\cpt$. 
\end{rem}

\section{Higher nonunital Quillen $\K'$-theory} \label {KQ}
Let $k$ be a field and let $\kAlg$ denote the category of $k$-algebras with not necessarily unital $k$-algebra homomorphisms. A module $M$ over a $k$-algebra $R$ will always mean a $k$-algebra module, i.e., $M$ has an underlying $k$-linear space with $k$-linear structure maps. A homomorphism of such modules will be both $k$-linear and $R$-linear by definition. We briefly recall a construction of Quillen \cite{QuiNonunitalK0}, which plays a central role in this article. Given any $k$-algebra $R$, embedded as a two-sided ideal in a unital $k$-algebra $S$, we consider the category $\Pr(S,R)$ whose objects are cochain complexes $X$ of right $S$-modules, which satisfy the following two conditions:

\begin{enumerate}
\item \label{cond1} each $X$ is homotopy equivalent to a strictly perfect complex, where a strictly perfect complex is a bounded complex of finitely generated and projective modules,
\item \label{cond2} the canonical map $X\otimes_{S}R\map X$ is a homotopy equivalence or, equivalently, $X/XR$ is an acyclic complex (see Proposition 2.2 of ibid. for some other equivalent descriptions of this condition). 
\end{enumerate}

\noindent
If $S=\tilde{R}$, i.e., the unitization of $R$, then we denote $\Pr(S,R)$ simply by $\Pr(R)$. We enrich the category $\Pr(R)$ over cochain complexes as explained in Example \ref{DGex} to make it a $k$-linear DG category. For a unital $k$-algebra $R$, the category $\Pr(R)$ may be regarded as a DG enrichment of the $k$-linear category of {\it homotopy perfect} complexes, i.e., complexes which are homotopy equivalent to strictly perfect complexes. Since every $P\in\Z^0(\Pr(R))$ is homotopy equivalent to a strictly perfect complex, by a result of Neeman \cite{NeeCompObj} it is a compact object in the derived category of cochain complexes of right $\tilde{R}$-modules. 

Any $k$-algebra homomorphism $g:R\map R'$ between possibly nonunital $k$-algebras extends uniquely to a unital map (preserving the adjoined unit) $\tilde{R}\map\tilde{R'}$ between their unitizations. Then one can consider $\tilde{R'}$ as a unital $k$-symmetric $\tilde{R}$-$\tilde{R'}$-bimodule (left structure is given by the unitization of ${g}$), which gives rise to a functor $g_*:= -\otimes_{\tilde{R}} \tilde{R'}:\Pr(R)\map\Pr(R')$. Since $g_*$ preserves cochain homotopy equivalences, one can check that condition \ref{cond1} above is preserved by it. In order to check condition \ref{cond2} we may restrict to strictly perfect complexes. Let $X\in\Pr(R)$ be a strictly perfect complex, whence $X\otimes_{\tilde{R}}\tilde{S}$ is a strictly perfect complex of right $\tilde{S}$-modules. Now $X\otimes_{\tilde{R}}\tilde{S} /(X\otimes_{\tilde{R}} \tilde{S})S \cong X\otimes_{\tilde{R}}\tilde{S}\otimes_{\tilde{S}} (\tilde{S}/S)\cong X\otimes_{\tilde{R}} k\cong X\otimes_{\tilde{R}}(\tilde{R}/R)\cong X/XR$ and $X/XR$ is acyclic, since $X\in\Pr(R)$. Therefore, we obtain a unital base change functor $\Pr:\kAlg\functor\NCS$ (see Remark \ref{functoriality}). Evidently, our construction $A\mapsto \Pr(A)$ is also functorial with respect to $*$-homomorphisms between $C^*$-algebras and produces a functor $\Pr:\CAlg\functor\NCS$. Composing it with the functor $\NCS\functor\DGcorr$ we get an additive category valued functor. We shall also denote the restricted functor to the full subcategory $\Csep$ by $\Pr$. Let us state one of the key results of \cite{QuiNonunitalK0} in the following Remark, which will be used frequently.

\begin{rem} \label{MorInv}
Let $S, S'$ be two unital $k$-algebras containing two-sided ideals $R, R'$ respectively. Let $\left(\begin{smallmatrix} S & M\\ M' & S'\end{smallmatrix}\right)$ be a Morita context such that the $S-S$-linear span of the image $p:M\otimes_{S'} M'\map S$ is $R\subset S$ and the $S'-S'$-linear span of the image $q:M'\otimes_S M\map S'$ is $R'\subset S'$. Then the functors $-\otimes_S M:\Pr(S,R)\map\Pr(S',R')$ and $-\otimes_{S'} M':\Pr(S',R')\map\Pr(S,R)$ are quasi-inverses between categories $\H^0(\Pr(S,R))$ and $\H^0(\Pr(S',R'))$. This implies that $-\otimes_S M$ and $-\otimes_{S'} M'$ are mutual inverses in $\NCS$, since $\H^n(\Hom(X,Y))\cong\H^0(\Hom(X,Y[n]))\overset{\sim}{\map}\H^0(\Hom((X\otimes_S M),(Y\otimes_S M)[n]))\cong\H^n(\Hom(X\otimes_S M, Y\otimes_S M))$. Furthermore, for any $X\in\Pr(S,R)$ and $Y\in\Pr(S',R')$, the canonical maps $X\otimes_{S}M\otimes_{S'} M'\map X$ and 
$Y\otimes_{S'}M'\otimes_S M\map Y$, which induce natural transformations $-\otimes_{S'}M'\otimes_S M\map \Id_{\Pr(S,R)}$ and $-\otimes_{S'}M'\otimes_S M\map \Id_{\Pr(S',R')}$ respectively, are homotopy equivalences (see Theorem 3.1 and Corollary 3.2 of \cite{QuiNonunitalK0}).

\noindent
This invariance result is stronger than Morita invariance implemented by equivalence bimodules. If $R=S$ and $R'=S'$, then one obtains the usual Morita invariance.
\end{rem}

\begin{lem} \label{RSR}
Let $0\map R\map S\map T\map 0$ be an exact sequence in $\kAlg$, where $S$ and $T$ are unital. Then $\Pr({S},R)\cong\Pr(R)$ in $\NCS$.
\end{lem}

\begin{proof}
There is a morphism of exact sequences in $\kAlg$

\beqn
\xymatrix{
0\ar[r] & R\ar[r] \ar[d]_{\id} & \tilde{R}\ar[r] \ar[d] & k\ar[r] \ar[d] & 0\\
0\ar[r] & R\ar[r] & {S}\ar[r] & {T}\ar[r] & 0
}
\eeqn One uses the Morita context $\left(\begin{smallmatrix} \tilde{R} & {S}\\ R & {S}\end{smallmatrix}\right)$ to construct maps $-\otimes_{{S}}R:\Pr(S,R)\map\Pr(R)$ and $-\otimes_{\tilde{R}} {S}:\Pr(R)\map\Pr({S},R)$. From Remark \ref{MorInv} we conclude that $-\otimes_{{S}}R$ and $-\otimes_{\tilde{R}}{S}$ are mutual inverses between $\Pr(R)$ and $\Pr(\tilde{S},R)$ in $\NCS$.
\end{proof}

\noindent
A functor from $\kAlg$ to any category is called {\it $\MM_2$-stable} if the image of the corner embedding $R\map \MM_2(R)$ sending $r\mapsto\left(\begin{smallmatrix} r & 0\\ 0 & 0\end{smallmatrix}\right)$ is an isomorphism.

\begin{lem} \label{MStable}
The functor $\Pr:\kAlg\functor\NCS$ is $\MM_2$-stable.
\end{lem}

\begin{proof}
For any $A\in\kAlg$, set $M=\MM_2(\tilde{A})$. Let $M'$ be the unitization of $\MM_2(A)$ with the canonical map $M'\overset{\theta}{\map} M$. Consider the Morita context $\left(\begin{smallmatrix}\tilde{A} & Q\\ P & M\end{smallmatrix}\right)$, where $Q=\left(\begin{smallmatrix}{A} & {A}\\ 0 & 0\end{smallmatrix}\right)$ and $P=\left(\begin{smallmatrix}\tilde{A} & 0\\ \tilde{A} & 0\end{smallmatrix}\right)$. Here the left action of $\tilde{A}$ on $Q$ and the right action of $\tilde{A}$ on $P$ are given by the map $\tilde{A}\map M'\overset{\theta}{\map} M$ where the first map comes from the unitization of the corner embedding $A\overset{\iota}{\map}\MM_2(A)$. The $M-M$ bimodule generated by the image of the map $P\otimes_{\tilde{A}}Q\overset{\gamma}{\map} M$ sending $p\otimes q\mapsto pq$ is $\MM_2(A)$. Clearly the image lies inside $\MM_2(A)$.  Let $e_{ij}$ denote the $\tilde{A}$-valued $2\times 2$ matrix with $1\in\tilde{A}$ as its $ij$-th entry and elsewhere $0$. Given any $\left(\begin{smallmatrix}x & y\\ z & w\end{smallmatrix}\right)\in\MM_2(A)$, write it as $\left(\begin{smallmatrix}x & y\\ z & w\end{smallmatrix}\right)=(e_{11})(xe_{11}) +(e_{11})(ye_{12}) + (e_{21})(ze_{11}) +(e_{21})(we_{12})$, which shows that it is in the image of $\gamma$. Similarly, one checks that the $\tilde{A}-\tilde{A}$ bimodule generated by the image of the map $Q\otimes_M P \map \tilde{A}$ is $A$. It remains to apply the above Remark \ref{MorInv}  and Lemma \ref{RSR} to obtain the desired isomorphism $\Pr(A)\cong\Pr(M,\MM_2(A))\cong\Pr(\MM_2(A))$. 
\end{proof}

\begin{prop} \label{StabHtpy}
The functor $\Pr:\kAlg\functor\NCS$ is $\MM_n$-stable for all $n\in\NN$.
\end{prop}

\begin{proof}
The assertion follows from the previous Lemma \ref{MStable} since, if a functor from $\kAlg$ to any category is $\MM_2$-stable, then it is automatically $\MM_n$-stable for all $n\in\NN$ (see Lemma 3.10 of \cite{CunMeyRos}). 
\end{proof}

Let $\cC$ be a category with a zero object $\ast$. Let there be two chosen subcategories $w\cC$ (subcategory of {\it weak equivalences}) and $co\cC$ (subcategory of {\it cofibrations}), such that every isomorphism of $\cC$ is a morphism in both $w\cC$ and $co\cC$. Suppose $co\cC$, in addition, satisfies the conditions:

\begin{enumerate}
\item For every object $C\in\cC$, the unique map $\ast\map C$ is a morphism in $co\cC$,

\item If $C\map D$ is a map in $co\cC$ and $C\map E$ is any map in $\cC$, then the pushout $D\sqcup_C E$ exists in $\cC$ and the canonical map $E\map D\sqcup_C E$ is in $co\cC$.
\end{enumerate}

The triple $(\cC,w\cC,co\cC)$ is called a {\it Waldhausen category structure on $\cC$} (or simply a Waldhausen category) if these data satisfy one more axiom `gluing lemma' (see Section 1.2 of \cite{Waldhausen}). A functor $F:\cC\map\cD$ between two Waldhausen categories is called {\it Waldhausen exact} if $F(co\cC)\subset co\cD$, $F(w\cC)\subset w\cD$, $F(\ast_\cC)\cong \ast_\cD$ and the canonical map $F(D)\sqcup_{F(C)} F(E)\map F(D\sqcup_C E)$ is an isomorphism whenever $C\map E$ is in $co\cC$.

\begin{ex}
Any Quillen exact category gives rise to a Waldhausen category with the admissible monomorphisms as cofibrations and the isomorphisms as weak equivalences.
\end{ex}

Waldhausen constructed a $\K$-theory, which takes any (small) Waldhausen category as input and produces a spectrum $\bK^{w}$. More precisely, $\bK^{w}$ is a functor from the category of Waldhausen categories with Waldhausen exact functors to the (triangulated) homotopy category of spectra $\HoSpt$. We refer the readers to  \cite{Waldhausen} for the details. This construction gives rise to a functor $\bK_{dg}:\DGcat\functor\HoSpt$. Given any DG category $\cC$ one constructs a Waldhausen category structure on the category of perfect $\cC$-modules $\Z^0(\per_{dg}(\cC))$ with cofibrations as module morphisms, which admit a graded splitting, and weak equivalences as quasi-isomorphisms. The homotopy category $\H^0(\per_{dg}(\cC))=\per(\cC)$ carries a triangulated category structure with the distinguished triangles being those, induced by the graded split exact sequences of perfect modules. Then one applies $\bK^w$ to the Waldhausen category $\Z^0(\per_{dg}(\cC)))$ to construct the $\K$-theory spectrum $\bK_{dg}(\cC)$. So, by definition, $\bK_{dg}(\cC)=\bK^{w}(\Z^0(\per_{dg}(\cC)))$ and the functor $\bK_{dg}:\DGcat\functor\HoSpt$ factors through $\NCS$ (and even $\DGcorr$, see \cite{Tab3,DugShi}). 

\begin{defn} \label{NKQ}
For any $k$-algebra $R$ we define {\it Quillen's nonunital (connective) $\K$-theory spectrum} of $R$, denoted by $\kQ(R)$, to be $\bK_{dg}(\Pr(R))$.
\end{defn}

\begin{rem}
One can also define a nonconnective version of $\kQ$ using Schlichting's techniques involving `Frobenius pairs' \cite{SchNegK}, but we do not investigate that theory here. One can replace $\K_0(\rep_{dg}(\cC,\cD))$ in the definition of morphisms of $\DGcorr$ by the nonconnective $\K$-theory spectrum of $\rep_{dg}(\cC,\cD)$ to obtain a variant of it, weakly enriched over the homotopy category of spectra and admitting a triangulated category structure (see \S 10 ibid.). The nonconnective version of $\kQ$ will be better related to this nonconnective variant of $\DGcorr$.
\end{rem}

Quillen defined $\K'_0(R)$ for a (possibly) nonunital ring $R$ as the free abelian group generated by the homotopy classes $[X]$ of the objects of $\Pr(\tilde{R},R)$ subject to relations $[X]=[X']+[X'']$, whenever $X'\map X\map X''$ is a short exact sequence of complexes, which is split in each degree (see \S 4 of \cite{QuiNonunitalK0}). We shall show below that $\pi_0(\kQ(R))=\KQ_0(R)$ agrees with Quillen's definition of nonunital $\K'_0$ of a possibly nonunital $k$-algebra $R$.

\noindent
As a immediate consequence of Proposition \ref{StabHtpy} one deduces the matrix stability of $\kQ$.

\begin{prop} \label{KMorInv}
The functor $\kQ:\kAlg\functor\HoSpt$ is $\MM_n$-stable for all $n\in\NN$. 
\end{prop}

Let $S$ be a unital $k$-algebra and $R$ be a two-sided ideal in $S$. There is a Waldhausen category structure on $\Z^0(\Pr(S,R))$, where the weak equivalences are precisely the homotopy equivalences and the cofibrations are the {\it graded (or locally) split monomorphisms}, i.e., maps of cochain complexes, which are split monomorphisms in each degree. Let us endow $\Z^0(\Pr(R)):=\Z^0(\Pr(\tilde{R},R))$ with this Waldhausen category structure. The cofibration sequences in this Waldhausen category are all graded split exact sequences of complexes. 

\begin{defn}
For any $k$-algebra $R$, let us denote by $\bK'(R)$ the Waldhausen $\K$-theory spectrum of $\Z^0(\Pr(R))$. 
\end{defn} We shall be interested mostly in the following specific type of Waldhausen categories. A Waldhausen category is {\it permissible} in the sense of Thomason (see \cite{NeeRan}) if it is a full subcategory of a category of cochain complexes over an abelian category such that the cofibrations are graded split monomorphisms and weak equivalences contain quasi-isomorphisms. Both $\Z^0(\Pr(S,R))$ and $\Z^0(\per_{dg}(\Pr(S,R)))$ are permissible Waldhausen categories, whose homotopy categories are $\H^0(\Pr(S,R))$ and $\H^0(\per_{dg}(\Pr(S,R)))=\per(\Pr(S,R))$ respectively. The category $\Z^0(\per_{dg}(\Pr(S,R)))$ can be viewed as a full subcategory of cochain complexes over the abelian category $\textup{Fun}(\Pr(S,R),\textup{Vec}_k)$, i.e., the abelian functor category of $k$-linear functors from $\Pr(S,R)$ to $k$-linear spaces. We know that a bounded below acyclic complex of projective modules is contractible. A morphism in $\Z^0(\Pr(S,R))$ is a quasi-isomorphism if and only if it is a cochain homotopy equivalence, since every complex in it is homotopy equivalent to a strictly perfect complex. Observe that both categories are also {\it saturated} Waldhausen categories, i.e., the weak equivalences satisfy the condition: if $f, g$ and $fg$ are morphisms in the Waldhausen category; whenever any two of them are weak equivalences, then so is the third.

\noindent
Obviously the category $\H^0(\Pr(S,R))$ is closed under translations (in both directions). Let us declare the triangles, induced by the graded split short exact sequences of complexes in $\Z^0(\Pr(S,R))$, to be distinguished. This makes $\H^0(\Pr(S,R))$ into a triangulated subcategory of the homotopy category of complexes of right $S$-modules. The usual mapping cone of a morphism of complexes acts as the cone of a morphism in $\H^0(\Pr(S,R))$. Any $k$-algebra homomorphism $g:R\map R'$ induces an exact functor of triangulated categories $\H^0(\Pr(R))\map\H^0(\Pr(R'))$, i.e., $-\otimes_{\tilde{R}}\tilde{R'}$ is additive and preserves graded split exact sequences. Similarly, $\H^0(\per_{dg}(\Pr(S,R)))$ carries a triangulated category structure induced by the graded split short exact sequences of perfect $\Pr(S,R)$-modules.

\noindent
In the sequel let us denote a homotopy equivalence of complexes (resp. quasi-isomorphism of complexes) by $\simeq$ (resp. $\simeq_{qi})$, i.e., $X\simeq Y$ (resp. $X\simeq_{qi} Y$) mean that there is a homotopy equivalence (resp. quasi-isomorphism) between $X$ and $Y$, but the choice is not specified.

\begin{lem} 
For any two sided ideal $R$ of a unital $k$-algebra $S$, the Yoneda functor $h:\Pr(S,R)\map\per_{dg}(\Pr(S,R))$ induces a Waldhausen exact functor $h:\Z^0(\Pr(S,R))\map \Z^0(\per_{dg}(\Pr(S,R)))$. 
\end{lem}

\begin{proof}
It is clear that the induced functor $h:\Z^0(\Pr(S,R))\map \Z^0(\per_{dg}(\Pr(S,R)))$ preserves graded split monomorphisms and sends cochain homotopy equivalences to quasi-isomorphisms. The pushout of the diagram $Y\hookleftarrow X\map Z$, $Y\hookleftarrow X$ being a cofibration, is described by the graded split exact sequence in $\Z^0(\Pr(R))$
\beqn
0\map X\map Y\oplus Z\map Y\oplus_X Z\map 0
\eeqn Applying $\Hom(P,-)$ to the above sequence one deduces that the canonical map 

\beqn
\Hom(P,Y)\oplus_{\Hom(P,X)}\Hom(P,Z)\map \Hom(P,Y\oplus_X Z)\eeqn is an isomorphism. This shows that $h(X)=\Hom(-,X)$ is Waldhausen exact.
\end{proof}

\begin{lem} \label{MorFib}
For any two sided ideal $R$ of a unital $k$-algebra $S$, the Waldhausen exact embedding $h:\Z^0(\Pr(S,R))\map\Z^0(\per_{dg}(\Pr(S,R)))$ induces an equivalence of triangulated homotopy categories $h:\H^0(\Pr(S,R))\map\per(\Pr(S,R))$.
\end{lem}

\begin{proof}
Since $h:\Z^0(\Pr(S,R))\map\Z^0(\per_{dg}(\Pr(S,R)))$ is a fully faithful exact functor between permissible Waldhausen categories, by Section 1.9.7 of \cite{ThoTro}, we need to check:

\begin{enumerate} 
\item \label{check1} any morphism $f$ in $\Z^0(\Pr(S,R))$ is a homotopy equivalence if and only if the morphism $h(f)$ in $\Z^0(\per_{dg}(\Pr(S,R)))$ is a quasi-isomorphism, and
\item \label{check2} the induced $h:\H^0(\Pr(S,R))\map\per(\Pr(S,R))$ is essentially surjective.
\end{enumerate}

\eqref{check1} Obviously $X\simeq Y$ implies that $h(X)=\Hom(-,X)\simeq_{qi} \Hom(-,Y)=h(Y)$. Conversely, suppose $\Hom(-,X)\simeq_{qi} \Hom(-,Y)$, i.e., there is a map $f:X\map Y$ that induces a quasi-isomorphism $\Hom(P,X)\simeq_{qi}\Hom(P,Y)$ for all $P\in\Z^0(\Pr(S,R))$. Putting $P=Y$, we obtain an isomorphism $\H^0(\Hom(Y,X))\overset{\H^0(f)}{\map}\H^0(\Hom(Y,Y))$. Let the class $[g]$ be the preimage of $[id]\in\H^0(\Hom(Y,Y))$ under this isomorphism. Then $g$ is a right homotopy inverse of $f$ and, similarly, putting $P=X$ one obtains a left homotopy inverse of $f$.

\eqref{check2} The objects of $\per(\Pr(S,R))$ are by definition those obtained by translations, formation of mapping cones and passage to direct summands of the image $h(\H^0(\Pr(S,R)))$ in the derived category $\textup{D}(\Pr(S,R))$. Clearly $h$ commutes with translations (in both directions) and formations of mapping cones. This implies that the image of the induced map $h:\H^0(\Pr(S,R))\map\per(\Pr(S,R))$ is cofinal in the sense that, given any $Y\in\per(\Pr(S,R))$ there is an $X\in\H^0(\Pr(S,R))$ such that $Y$ is a direct summand of $h(X)$. Now let us verify that the category $\H^0(\Pr(S,R))$ is idempotent complete. Given any arrow $p:X\map X$ in $\H^0(\Pr(S,R))$, such that $p^2=p$, one can write $X\simeq pX\oplus (\id-p)X$ in the larger homotopy category of all unbounded cochain complexes of right modules over $S$. We need to check the $pX$ is an object of $\H^0(\Pr(S,R))$. By Proposition 1.1. of \cite{QuiNonunitalK0} a complex is homotopy equivalent to a strictly perfect complex if and only if it is a homotopy retract of a strictly perfect complex. It is clear that $pX$ is a homotopy retract of $X$, which is itself a homotopy retract of a strictly perfect complex. The condition $pX\otimes_S R\cong pX$ follows easily from the decomposition $X\simeq pX\oplus (\id-p)X$ and the fact that $-\otimes_S R$ is an additive functor.
\end{proof}

\begin{lem} \label{comp}
For all $R\in\kAlg$, the exact functor $h:\Z^0(\Pr(R))\map\Z^0(\per_{dg}(\Pr(R)))$ induces a natural weak equivalence of spectra $\bK'(R)\overset{\sim}{\map}\kQ(R)$. Furthermore, $\pi_0(\kQ(R))\cong\pi_0(\bK'(R))=\K'_0(R)$, where $\K'_0$ on the right hand side denotes Quillen's original definition of nonunital $\K_0$ of a ring.
\end{lem}

\begin{proof}
Since the canonical induced map $\H^0(\Pr(R))\map\per(\Pr(R))$ is an exact equivalence of triangulated categories (see Lemma \ref{MorFib} above), by Grayson's Cofinality Theorem as stated in Theorem 2.4 of \cite{NeeRan} the induced map of spectra $\bK'(R)\map\kQ(R)$ produces an isomorphism of homotopy groups $\pi_i$ for all $i\geqslant 0$, i.e., it is a weak equivalence of spectra. The naturality of this weak equivalence is clear.

The second assertion follows easily from the description of the cofibrations and the weak equivalences in the Waldhausen category structure on $\Z^0(\Pr(R))$ and the general description of the $\K_0$-group of any Waldhausen category.
\end{proof}

\noindent
Now we can set $\K'_i(R)=\pi_i(\bK'(R))$ and $\KQ_i(R)=\pi_i(\kQ(R))$ for all $i\geqslant 0$; in addition, we know that $\KQ_0(R)\cong\K'_0(R)$ in agreement with Quillen's original definition. Each element in $\K'_0(R)$ can be described by the class of a two-term complex $f:P^0\map P^1$, where $P^i$ are finitely generated and projective right $\tilde{R}$-modules and $f$ is an $\tilde{R}$-module homomorphism, which induces an isomorphism modulo $R$, $\overline{f}:P^0/P^0R\map P^1/P^1R$. There is a surjective group homomorphism $\nu_0:\K'_0(R)\map\K^\alg_0(R)$, sending the class $[f:P^0\map P^1]\in\K'_0(R)$ to the class $[P^1]-[P^0]\in\K^\alg_0(R):=\ker[\K_0^\alg(\tilde{R})\map\K_0^\alg(k)]$ (see \S 6 of \cite{QuiNonunitalK0}).

Let $\IAlg$ denote the full subcategory of $\kAlg$, consisting of {\it idempotent} $k$-algebras $R$, i.e., $R$ satisfying $R^2=R$, where $R^2=\{\sum_{i=1}^n r_ir'_i\,|\, r_i, r'_i\in R\}$. Any unital $k$-algebra is evidently an idempotent algebra. Any $C^*$-algebra is also an idempotent algebra, which follows from the Cohen--Hewitt factorization Theorem (see Theorem 2.5 of \cite{Hewitt}). It is clear that, if $R$ is a unital $k$-algebra then, for any $R'\in\IAlg$, one finds $R\otimes_k R'\in\IAlg$. Consequently, for any $R\in\IAlg$, $\MM_n(R)=\MM_n(k)\otimes_k R$ is an idempotent $k$-algebra.

\begin{lem} \label{HomotopyFibre}
Let $0\map R\map S\map T\map 0$ be an exact sequence in $\IAlg$. Then, for any $X\in\Pr(S)$, the complex $X\otimes_{\tilde{S}}\tilde{T}$ is homotopy equivalent to $0$ in $\Pr(T)$ if and only if $X\in\Pr(\tilde{S},R)$.
\end{lem}

\begin{proof}
For any $X\in\Pr(\tilde{S},R)$, let us show that $X\otimes_{\tilde{S}}\tilde{T}$ is homotopy equivalent to $0$ in $\Pr(T)$. We know that $X\otimes_{\tilde{S}} R\simeq X$, whence $X\otimes_{\tilde{S}}R\otimes_{\tilde{S}}\tilde{T}\simeq X\otimes_{\tilde{S}}\tilde{T}$. Now any element $r\otimes t\in R\otimes_{\tilde{S}}\tilde{T}$  can be written as $\sum_{i=1}^n r'_ir''_i\otimes t$, since $R^2=R$. Each $r'_ir''_i\otimes t= r'_i\otimes r''_it = 0$, i.e., $R\otimes_{\tilde{S}}\tilde{T}$ is trivial as an $\tilde{S}-\tilde{T}$-bimodule. This shows that $X\otimes_{\tilde{S}}\tilde{T}\simeq 0$. 

\noindent
Conversely, consider the short exact sequence of $\tilde{S}-\tilde{S}$-bimodules

\beqn
0\map R\map {\tilde{S}}\map \tilde{T}\map 0.
\eeqn Without loss of generality assume that $X\in \Pr(S)$ is a strictly perfect complex, which gives rise to the following short exact sequence of complexes of right $\tilde{S}$-modules

\beqn
0\map X\otimes_{\tilde{S}} R\map X\otimes_{\tilde{S}}\tilde{S}\cong X\map X\otimes_{\tilde{S}}\tilde{T}\map 0
\eeqn Since $X$ is a strictly perfect complex, $XR\cong X\otimes_{\tilde{S}} R$ (isomorphic as complexes of right $\tilde{S}$-modules), whence $X/XR\cong X\otimes_{\tilde{S}}\tilde{T}$ is acyclic. By Proposition 2.2 of \cite{QuiNonunitalK0} the map $X\otimes_{\tilde{S}} R\map X$ must be a homotopy equivalence, i.e., $X\in\Pr(\tilde{S},R)$.
\end{proof}

\noindent
A sequence of DG functors $\cA\overset{i}{\map}\cB\overset{j}{\map}\cC$ such that $ji=0$ in $\NCS$ is called {\it exact}, if $j$ is the cokernel of $i$ and $i$ is the kernel of $j$.

\begin{lem} \label{DGExact} \footnote{see also Erratum available at the author's homepage}
Let $0\map R\overset{\psi}{\map} S\overset{\phi}{\map} T\map 0$ be any exact sequence in $\kAlg$ such that $R\in\IAlg$ and $S,T$ are unital. Then the induced diagram $\Pr(R)\map\Pr(S)\map\Pr(T)$ is an exact sequence.
\end{lem}

\begin{proof}
Using Lemma \ref{RSR} one obtains a diagram $\Pr(R)\cong\Pr(\tilde{S},R)\overset{\psi}{\map}\Pr(S)\overset{\phi}{\map}\Pr(T)$ in $\NCS$, where $\psi=-\otimes_{\tilde{R}}\tilde{S}$ and $\phi=-\otimes_{\tilde{S}}\tilde{T}$. It follows from $R^2=R$ that $\phi\circ\psi \simeq 0$, since $R\otimes_{\tilde{R}}\tilde{S}\otimes_{\tilde{S}}\tilde{T}\cong R\otimes_{\tilde{R}}\tilde{T}\cong 0$. In order to show that $\Pr(\tilde{S},R)\map\Pr(S)\map\Pr(T)$ is an exact sequence, we need to show that $\H^0(\Pr(\tilde{S},R))$ is a thick subcategory of $\H^0(\Pr(S))$ and the induced map $\H^0(\Pr(S))/\H^0(\Pr(\tilde{S},R))\overset{\phi}{\map}\H^0(\Pr(T))$ after Verdier localization is essentially surjective up to factors (see Theorem 4.11 of \cite{KelDG}). Using Lemma \ref{MorFib} and Lemma \ref{HomotopyFibre} we find that $\H^0(\Pr(\tilde{S},R))$ is the kernel of the map $\phi$ and hence thick. The induced functor $\H^0(\Pr(S))/\H^0(\Pr(S,R))\overset{\phi}{\map} \H^0(\Pr(T))$ is exact. 

Let us now show that the induced functor $\H^0(\Pr(S))/\H^0(\Pr(S,R))\overset{\phi}{\map} \H^0(\Pr(T))$ is fully faithful. Choose any $g\in\Hom_{\Pr(S)}(X,Y)$ such that $\phi(g)=0$. In the abelian category of bounded cochain complexes of all right $\tilde{S}$-modules one has the following commutative diagram

\beqn
\xymatrix{
X\ar[r]^g \ar[d] &\textup{Im}(g) \ar[d] \\
X\ar[r]^g  & Y,
}
\eeqn which gives rise to a morphism of distinguished triangles in the bounded homotopy category $K^b(\tilde{S})$ of cochain complexes

\beqn
\xymatrix{
Z[-1]\ar[r]^\alpha\ar[d] &X\ar[r]^g\ar[d] &\textup{Im}(g)\ar[r]\ar[d] &Z\ar[d] \\
Z'[-1]\ar[r] & X\ar[r]^g & Y\ar[r] & Z'
}
\eeqn The functor $\phi$ naturally extends to an exact functor $\bar{\phi}:=-\otimes^\mathbb{L}_{\tilde{S}}\tilde{T}:K^b(\tilde{S})\map K^b(\tilde{T})$ and $\bar{\phi}(\Im(g))\simeq 0$. Since $\Cone(\alpha)= \textup{Im}(g)\in\ker(\bar{\phi})$, $\alpha$ is invertible in the Verdier localization $K^b(\tilde{S})/\ker(\bar{\phi})$. Now $g\circ\alpha = 0$ implies that $g$ factors through $0$ in $K^b(\tilde{S})/\ker(\bar{\phi})$. Since $\H^0(\Pr(S))/\ker(\phi)$ is a subcategory of $K^b(\tilde{S})/\ker(\bar{\phi})$, one concludes that $g=0$ in $\H^0(\Pr(S))/\ker(\phi)$ whence the induced functor $\phi$ is faithful.

Now we need to prove that the induced map $\Hom(X,Y)\map\Hom(X\otimes_{\tilde{S}}\tilde{T},Y\otimes_{\tilde{S}}\tilde{T})$ is surjective. Without loss of generality, let us assume that $X,Y$ are strictly perfect complexes. Since any strictly perfect complex is a homotopy retract of a bounded complex of finitely generated and free right modules (see Proposition 1.1 of \cite{QuiNonunitalK0} and also \cite{Ranicki}), we may further assume that $X,Y$ consist of finitely generated and free right $\tilde{S}$-modules in each degree. Choose any $\tilde{T}$-linear map of cochain complexes $f:X\otimes_{\tilde{S}}\tilde{T}\map Y\otimes_{\tilde{S}}\tilde{T}$ and consider the diagram in the abelian category of cochain complexes over $\tilde{S}$

\beqn
\xymatrix{
X\otimes_{\tilde{S}}\tilde{S}\cong X\ar[d] & Y\cong Y\otimes_{\tilde{S}}\tilde{S}\ar[d]\\
X\otimes_{\tilde{S}}\tilde{T}\ar[r]^f & Y\otimes_{\tilde{S}}\tilde{T},
}
\eeqn where $f$ is viewed as a morphism of complexes over $\tilde{S}$ via the restriction of scalars. We need to construct a map $\bar{f}:X\map Y$ of complexes over $\tilde{S}$, making the above diagram commute. This can be done by induction on the length of $X$, using the fact that the surjection $M_{n\times m}(\tilde{S})\map M_{n\times m}(\tilde{T})$ induced by $\tilde{\phi}:\tilde{S}\map\tilde{T}$ admits a $k$-linear splitting (since $k$ is a field).

It remains to show that $\phi$ is essentially surjective up to factors, i.e., given any $Y\in\H^0(\Pr(T))$, there is an $X\in\H^0(\Pr(S))$, such that $Y$ is a direct summand of $\phi(X)$. We may assume that $Y\in\H^0(\Pr(T))$ is a strictly perfect complex. Then, as noted above, $Y$ is a homotopy retract of a complex $F$, where $F$ is a bounded complex of finitely generated and free right $\tilde{T}$-modules. It follows that $Y\otimes_{\tilde{T}} T\simeq Y$ is a direct summand of $F\otimes_{\tilde{T}} T$ in the idempotent complete category $\H^0(\Pr(T))$. One can verify by induction on the number of nonzero terms of $F$ that there is a bounded complex $F'$ of finitely generated and free right $\tilde{S}$-modules, such that $F'\otimes_{\tilde{S}}\tilde{T} \simeq F$. Since $T$ and $S$ are unital, they are finitely generated and projective $\tilde{T}$ and $\tilde{S}$ bimodules respectively, from which we get $F\simeq F\otimes_{\tilde{T}} T\in\Pr(T)$ and $F'\simeq F'\otimes_{\tilde{S}} S\in\Pr(S)$. Finally using $S\otimes_{\tilde{S}}\tilde{T}\cong T$, we obtain

\beqn
\phi(F'\otimes_{\tilde{S}} S)= F'\otimes_{\tilde{S}} S\otimes_{\tilde{S}}\tilde{T} \simeq F'\otimes_{\tilde{S}} T\simeq F'\otimes_{\tilde{S}} \tilde{T}\otimes_{\tilde{T}} T \simeq F\otimes_{\tilde{T}} T.
\eeqn\end{proof}

\noindent
Let $F:\kAlg\functor\HoSpt$ be any functor. We call a $k$-algebra $R$ {\it connectively $F$-excisive} if, whenever $0\map R\map S\map T\map 0$ is an exact sequence in $\kAlg$ with $S$ and $T$ unital, the induced diagram $F(R)\map F(S)\map F(T)$ is a {\it connective homotopy fibration}, i.e., it gives rise to a long exact sequence of homotopy groups

\beqn
\cdots \map \pi_1(F(R))\map\pi_1(F(S))\map\pi_1(F(T))\map\pi_0(F(R))\map\pi_0(F(S))\map\pi_0(F(T)).
\eeqn As a consequence of Theorem 5.1 of \cite{KelDG} we draw the following conclusion. 

\begin{prop} \label{AlgExcision}
Any idempotent $k$-algebra is connectively $\kQ$-excisive.
\end{prop} 

\begin{rem}
We do not know whether the idempotence condition $R^2=R$ is necessary for a $k$-algebra $R$ to be connectively $\kQ$-excisive. 
\end{rem}

For any unital $k$-algebra $S$ let $\Perf(S)$ denote the $k$-linear category of {\it perfect} complexes of right $S$-modules with cochain maps. Recall that a complex of right $S$-modules is called {\it perfect} if it is quasi-isomorphic to a strictly perfect complex. It becomes a Waldhausen category with quasi-isomorphisms as weak equivalences and graded split monomorphisms as cofibrations. For any $k$-algebra $R$, the canonical inclusion $\Z^0(\Pr(R))\map\Perf(\tilde{R})$ is an exact functor between Waldhausen categories. Due to the work of Gillet--Waldhausen, it is known that the Waldhausen $\K$-theory spectrum of $\Perf(\tilde{R})$ is homotopy equivalent to Quillen's algebraic $\K$-theory spectrum of $\tilde{R}$ obtained by $Q$-construction (see Theorem 6.2 of \cite{Gil1} and, also, Lemma 1.1 of \cite{WeiYao}). Henceforth, let $\bK^\alg$ denote the connective algebraic $\K$-theory spectrum. The canonical inclusion of Waldhausen categories induces a map $\nu:\kQ(R)\overset{\sim}{\leftarrow}\bK'(R)\map\bK^{\textup{alg}}(\tilde{R})$ in $\HoSpt$. When $R$ is a $C^*$-algebra, using Proposition 1.3.1 of \cite{Waldhausen} one finds that the image of the composition map $\bK'(R)\map\bK^\alg(\tilde{R})\overset{-\otimes^{\mathbb{L}}_{\tilde{R}}k}{\map}\bK^\alg(k)$ is trivial, since $X\otimes_{\tilde{R}}^{\mathbb{L}} k\simeq_{qi} 0$ for all $X\in\Pr(R)$, i.e., the composition map is weakly equivalent to the $0$ map. It follows that the image of $\pi_i(\nu)$ lies inside $\K^{\textup{alg}}_i(R):=\ker[\pi_i(\bK^{\textup{alg}}(\tilde{R}))\map\pi_i(\bK^{\textup{alg}}(k))]$. Since any $C^*$-algebra is (connectively) $\bK^\alg$-excisive (see Corollary 10.4 of \cite{SusWod2}), there is no ambiguity in talking about the algebraic $\K$-theory spectrum of a nonunital $C^*$-algebra. In fact, for any idempotent $k$-algebra $R$, which satisfies excision in algebraic $\K$-theory, there is an induced map of spectra $\nu:\bK'(R)\map\bK^{\alg}(R)$. The naturality of this map follows from the fact that tensoring with strictly perfect complexes preserves quasi-isomorphisms. At the level of $\pi_0$, the map is given by  $\nu_0=\pi_0(\nu):\KQ_0(R)\overset{\sim}{\leftarrow}\K'_0(R)\map\K^{\textup{alg}}_0(R):=\ker[\K^{\textup{alg}}_0(\tilde{R})\map\K^{\textup{alg}}_0(k)]$, where $\K^{\textup{alg}}_0(\tilde{R})$ (resp. $\K^{\textup{alg}}_0(k)$) can be identified with the Grothendieck group of stable isomorphism classes of finitely generated projective right modules over $\tilde{R}$ (resp. $k)$. Let $S$ be any unital $k$-algebra containing $R$ as a two-sided ideal. The group $\K^\alg_0(R)$ can also be described as the relative group $\K^\alg_0(S\map S/R)$ fitting into a long exact sequence of $\K^\alg$-theory, since $\K^\alg_0$ satisfies excision, i.e., $\K^\alg_0(S\map S/R)\cong\K^\alg_0(\tilde{R}\map k)$. In the paragraph after Lemma \ref{comp} it was mentioned that the map $\nu_0$ is always surjective (see also Proposition 8.1 of \cite{QuiNonunitalK0}).

\begin{rem} \label{QuillenIsom}
There is a natural isomorphism $\bK'\map\kQ$ (see Lemma \ref{comp}) and, on the category $\CAlg$ (or, more generally, on idempotent algebras, which satisfy excision in $\bK^{\alg}$), the map $\nu:\bK'\map\bK^\alg$ is a natural transformation. For the general excision problem in $\bK^{\alg}$ of $\QQ$-algebras, by a result of Corti{\~n}as \cite{CorExcision}, it suffices to check the vanishing of certain birelative algebraic cyclic homology groups, i.e., the obstruction to excision in cyclic homology. For any $R\in\IAlg$, the induced map $\nu_0$, as described above, is an isomorphism if $R$ is a $k$-algebra with local units (see Proposition 6.2 of \cite{QuiNonunitalK0}) or a $C^*$-algebra (see Proposition 6.3 of \cite{QuiNonunitalK0}).
\end{rem}

\section{Excision and topological $\K$-theory} \label{TopK}

\noindent
The definition of an {\it exact sequence} in the category of $C^*$-algebras $\CAlg$ is simply a diagram isomorphic to $0\map I\map A\map A/I\map 0$, where $I$ is a closed two-sided ideal in $A$. An exact sequence is further called {\it split exact} if it admits a splitting $*$-homomorphism $s: A/I\map A$. A functor from $\CAlg$ to an additive category is called {\it split exact} if it sends a split exact sequence of $C^*$-algebras to a direct sum diagram in the target category. Ignoring the involution, the $C^*$-norm and the compatibility of algebra homomorphisms with them, one can similarly define a {\it split exact} functor $F$ from $\IAlg$ to any additive category.

\begin{lem} \label{splitExact}
The functor $\Pr:\IAlg\functor\DGcorr$ is split exact.
 \end{lem}
 
\begin{proof}
For any split exact sequence

\begin{equation} \label{splitEqn}
\xymatrix { 0\ar[r]&A\ar[r]^i & B \ar[r]^j& C\ar[r]\ar@/_1pc/[l]_s& 0},
\end{equation} applying $\Pr$ we obtain the diagram

\beqn
\xymatrix { \Pr(A)\ar[r]^{I=i_*} & \Pr(B) \ar[r]^{J=j_*}& \Pr(C)\ar@/_2pc/[l]_{S=s_*}},
\eeqn where $JI=0$ and $JS= \id_{\Pr(C)}$. 

\noindent
Once again, consider the Morita context $\left(\begin{smallmatrix} \tilde{A} & \tilde{B}\\ A & \tilde{B}\end{smallmatrix}\right)$. Using Remark \ref{MorInv} it is readily verified that the functor $\kappa=-\otimes_{\tilde{B}}{A}:\Pr(B)\map\Pr(A)$ is the left inverse of $\kappa'=-\otimes_{\tilde{A}}\tilde{B}$ in $\DGcorr$, i.e., $\kappa\circ\kappa'=\id_{\Pr(A)}$. It suffices to show that $\kappa'\circ\kappa + SJ=\id_{\Pr(B)}$ in $\Hom_{\DGcorr}(\Pr(B),\Pr(B))$. There is diagram of natural transformations $\kappa'\circ\kappa\map\Id_{\Pr(B)}\leftarrow SJ$, such that for all $X\in\Pr(B)$, one has a diagram in $\H^0(\Pr(B))$ (see Lemma \ref{HomotopyFibre})

\beqn
\kappa'\circ\kappa(X)\hookrightarrow\Id_{\Pr(B)}(X)=X\hookleftarrow SJ(X)
\eeqn However, this diagram is homotopy equivalent to $X'A\otimes_{\tilde{A}} A\hookrightarrow X'\overset{S}{\leftarrow} X'/X'A$, where $X'\simeq X$ is a strictly perfect complex. Since $X'=X'A\otimes_{\tilde{A}} A + S(X'/X'A)$, one concludes that $\kappa'\circ\kappa +SJ =\id_{\Pr(B)}$ in $\DGcorr$. 
\end{proof}

\begin{cor} \label{splitCor}
The functor $\kQ:\IAlg\functor\HoSpt$ is split exact.
\end{cor}

\begin{defn} \label{excisDef}
We call a functor from $F:\IAlg\functor\HoSpt$ {\it connectively excisive} if, whenever $0\map R\map S\map T\map 0$ is an exact sequence  $\IAlg$, the induced diagram $F(R)\map F(S)\map F(T)$ is a connective homotopy fibration, i.e., it gives rise to a long exact sequence of homotopy groups

\beqn
\cdots \map \pi_1(F(R))\map\pi_1(F(S))\map\pi_1(F(T))\map\pi_0(F(R))\map\pi_0(F(S))\map\pi_0(F(T)).
\eeqn
\end{defn}

\begin{thm} \label{excision}
The functor $\kQ:\IAlg\functor\HoSpt$ is a connectively excisive.
\end{thm}

\begin{proof}
This result is a formal consequence Propostion \ref{AlgExcision} and Corollary \ref{splitCor}. Indeed, any exact sequence $0\map R\map S\map T\map 0$ in $\IAlg$ gives rise to another exact sequence $0\map R\map \tilde{S}\map \tilde{T}\map 0$. Using Proposition \ref{AlgExcision} we conclude that the induced map $\kQ(R)\map\kQ(\tilde{S})\map\kQ(\tilde{T})$ is a connective homotopy fibration. Since $\kQ$ is split exact, $\kQ(\tilde{S})\cong\kQ(S)\oplus\kQ(k)$ and $\kQ(\tilde{T})\cong\kQ(T)\oplus\kQ(k)$. From the following commutative diagram of maps of spectra

\beqn
\xymatrix{
\kQ(R)\ar[r]\ar[d] &\kQ(S)\ar[r]\ar[d] &\kQ(T)\ar[d]\\
\kQ(R)\ar[r] &\kQ(\tilde{S})\ar[r] &\kQ(\tilde{T})
}
\eeqn one can check that the diagram $\kQ(R)\map\kQ(S)\map\kQ(T)$ is also a connective homotopy fibration.
\end{proof}

\noindent
Let us now stabilize the functor $\Pr:\CAlg\functor\NCS$ by defining a new functor $\PrK(-)=\Pr(-\prot\cpt)$. Hence, by definition, for any $C^*$-algebra $A$, one finds $\PrK(A)=\Pr(A_\cpt)$.

\begin{lem} \label{KStable}
The functor $\PrK:\CAlg\functor\NCS$ is $C^*$-stable and, when restricted to $\Csep$, it factors through $\CSp$.
\end{lem}

\begin{proof}
The $C^*$-stability of $\PrK$ follows from Lemma \eqref{MStable}, since it is known that, if a functor $F$ from $\CAlg$ to any other category is $\MM_2$-stable, then the functor $F(-\prot\cpt)$ is $C^*$-stable (see Proposition 3.31 of \cite{CunMeyRos}). 

That the functor $\PrK$ factors through $\CSp$ follows from the characterization of $\Csep\functor\CSp$ as the universal $C^*$-stable functor (see Proposition 39 of \cite{Mey1}).
\end{proof}

\noindent
The induced functor $\PrK:\CSp\functor\NCS$ is a passage between noncommutative spaces in two different settings.

\begin{lem} \label{Htpy}
The functor $\PrK:\CAlg\functor\DGcorr$ is homotopy invariant and satisfies Bott periodicity. As a consequence $\K'_i(-\prot\cpt)$ is homotopy invariant for all $i\geqslant 0$.
\end{lem}

\begin{proof}
This is an immediate consequence of Theorem 3.2.2. of \cite{Hig2} which says that any $C^*$-stable and split exact functor on $\CAlg$ is automatically homotopy invariant and satisfies Bott periodicity. Observe that applying the functor $-\prot\cpt$ does not affect the split exactness of a diagram in $\CAlg$
\end{proof}

\noindent
Now we show that information about the topological $\K$-theory of a $C^*$-algebra can be extracted from the image of the functor $\PrK$. The proof is modelled along the lines of \cite{Hig2}. For any $C^*$-algebra $A$, let $\bK^{\textup{top}}(A)\langle 0\rangle$ denote the connective cover of the topological $\K$-theory spectrum of $A$. It is known that $\K^{\textup{top}}_0$ is a $C^*$-stable functor.

\noindent
For any $C^*$-algebra $A$, let $c^A:\bK^{\textup{alg}}(A)\map\bK^{\textup{top}}(A)\langle 0\rangle$ be the comparison map from the connective algebraic $\K$-theory to the connective cover of topological $\K$-theory at the level of spectra (see Theorem 2.1 of \cite{RosComparison}). The homotopy fibre $\hofib(c^A)$ is the obstruction to $c^A$ being an isomorphism. Precomposing $c^A$ with the map $\kQ(A)\overset{\nu^A}{\map}\bK^{\textup{alg}}(A)$ we get a comparison map $c^A_{\KQ}:\kQ(A)\map\bK^{\textup{top}}(A)\langle 0\rangle$. There is a subtlety in defining this composition. The natural recipient of the map $\nu^A$ is the Waldhausen model of $\bK^\alg(A)$, whereas the natural domain of the comparison map $c^A$ is the one arising from $\textup{BGL}^+(A)$ with discrete topology. One has to use the identification "$+=Q=S_{\bullet}$".

\begin{thm} \label{KSpec}
Let $A$ be a $C^*$-algebra and set $A_\cpt=A\prot\cpt$. Then $\bK_{dg}(\PrK(A))=\kQ(A_\cpt)$ is homotopy equivalent to $\bK^{\textup{top}}(A_\cpt)\langle 0\rangle$ via $c^{A_\cpt}_{\KQ}=c^{A_\cpt}\circ\nu^{A_\cpt}$. In particular, if $A$ is stable, then $\kQ(A)\cong\kQ(A_\cpt)\cong\bK^{\textup{top}}(A_\cpt)\langle 0\rangle\cong\bK^{\textup{top}}(A)\langle 0\rangle$. 
\end{thm}

\begin{proof}
The proof uses a standard dimension shifting argument for $C^*$-algebras using the `cone-suspension' exact sequence 

\begin{equation} \label{pathFib}
0\map\C_0((0,1))\prot A\map\C_0([0,1))\prot A\map A\map 0.
\end{equation}
 
Let us set $\C_0([0,1))\prot A) =\Cone A$ and $\C_0((0,1))\prot A = \Sigma A$. The functor $-\prot\cpt$ is exact, since $\cpt$ is a nuclear exact $C^*$-algebra. As discussed in Remark \ref{QuillenIsom}, the natural map $\kQ(A)\overset{\sim}{\leftarrow}
\bK'(A)\overset{\nu}{\map}\bK^{\textup{alg}}(A)$ in $\HoSpt$ induces an isomorphism at the level of $\pi_0$. Using Theorem \ref{excision} and the natural transformation $\nu_\cpt:\bK'(-\prot\cpt)\map\bK^\alg(-\prot\cpt)$, we obtain the following morphism of long exact sequences:

\beqn
\xymatrix{
\K'_1(\Cone A_\cpt)
\ar[r]
\ar[d]
& \K'_1(A_\cpt)
\ar[d]
\ar[r]
&\K'_0(\Sigma A_\cpt)
\ar[d]
\ar[r]
& \K'_0(\Cone A_\cpt)
\ar[d]
\ar[r]
&\K'_0(A_\cpt)
\ar[d] \\
 \K_1^{\textup{alg}}(\Cone A_\cpt) 
 \ar[r]
 &\K^{\textup{alg}}_1(A_\cpt)
 \ar[r]
 &\K^{\textup{alg}}_0(\Sigma A_\cpt)
 \ar[r]
 & \K_0^{\textup{alg}}(\Cone A_\cpt)
 \ar[r]
& \K^{\textup{alg}}_0(A_\cpt)\; .
 }
\eeqn We know that the three vertical arrows from the right are isomorphisms and $\Cone A_\cpt$ is contractible, i.e., homotopy equivalent to the zero $C^*$-algebra. Now we exploit the homotopy invariance of $\K'_i(-\prot\cpt)$ and $\K_i^{\textup{alg}}(-\prot\cpt)$ (see Proposition 10.6 of \cite{SusWod2}) to deduce that $\K'_i(\Cone A_\cpt)\cong\K_i^{\textup{alg}}(\Cone A_\cpt)\cong0$, whence the boundary maps $\K_1^{\textup{alg}}(A_\cpt)\map\K_0^{\textup{alg}}(\Sigma A_\cpt)$ and $\K'_1(A_\cpt)\map\K'_0(\Sigma A_\cpt)$ are isomorphisms. It follows immediately that $\K'_1(A_\cpt)\map\K_1^{\textup{alg}}(A_\cpt)$ is an isomorphism. Therefore, we have $\K'_i(A_\cpt)\map\K_i^{\textup{alg}}(A_\cpt)$ for $i=0,1$ and the isomorphisms for $i\geqslant 2$ follow similarly by induction on $i$. Thanks to the Theorem of Suslin--Wodzicki (see \cite{SusWod2}), we know that for a stable $C^*$-algebra (like $A_\cpt$) the algebraic $\K$-theory spectrum is (connectively) homotopy equivalent to the topological $\K$-theory spectrum via the comparison map $c^{A_\cpt}$ (see also, e.g., Theorem 3.3 of \cite{RosComparison}). Thus we have $\kQ(A_\cpt)\cong\bK^\alg(A_\cpt)\cong\bK^{\textup{top}}(A_\cpt)\langle 0\rangle$.
\end{proof}

\begin{rem}
Due to the Bott $2$-periodicity of the topological $\K$-theory groups, one can restrict one's attention to the Waldhausen $\K_0$ and $\K_1$-groups of $\Z^0(\Pr(A_\cpt))$. We already know how the $\K_0$-groups look like. For a description of the $\K_1$-groups we refer the readers to \cite{MurTon}. The above result actually exhibits $\kQ(A_\cpt)$ as an `explicit' model of a connective cover $\bK^{\textup{top}}(A_\cpt)\langle 0\rangle$.
\end{rem}

\noindent
Consider the following diagram in $\HoSpt$, whose rows are distinguished triangles

\beqn
\xymatrix{
\kQ(A)\ar[d]_{\nu^A}\ar[r]^{\iota} & \kQ(A_\cpt) \ar[d]^{\nu^{A_\cpt}}\ar[r] & \Cone(\iota)\ar[d]^\xi \\
\bK^{\textup{alg}}(A) \ar[r]^{\iota^{\textup{alg}}} & \bK^{\textup{alg}}(A_\cpt) \ar[r] & \Cone(\iota^\alg)
 }.
\eeqn Since $\nu^{A_\cpt}$ is an isomorphism (follows from Theorem \ref{KSpec}), using the $3\times 3$-Lemma we get that $\Cone(\xi)$ isomorphic to $\Sigma\Cone(\nu^A)$. If $\Cone(\iota)$ vanishes, then $\Cone(\xi)$ is isomorphic to $\Cone(\iota^\alg)$. The nonvanishing of $\Cone(\iota)$ (resp. $\Cone(\iota^\alg)$) would indicate the failure of $C^*$-stability of the functor $\KQ$ (resp.  $\K^\alg$). Although the $\KQ$-theory has better invariance properties than the $\K^\alg$-theory, it is not expected to be $C^*$-stable. Both $\Cone(\iota)$ and $\Cone(\iota^\alg)$ vanish on the category of stable $C^*$-algebras. 

\begin{rem}
The author cannot describe the homotopy fibre $\hofib(c^A_{\KQ})$ at the moment. It was shown in \cite{CorTho1} that, for a locally convex $\CC$-algebra $A$, stabilized by an operator ideal $A\prot\cJ$, the cofibre of the comparison map $c^{A\prot\cJ}$ is homotopy equivalent to the algebraic cyclic homology spectrum $\Sigma\bHC(A\prot\cJ)$. Relying on the vanishing results about $\bHC$, announced earlier by Wodzicki, the authors of ibid. could prove Karoubi's conjecture for a wide class of Fr{\'e}chet algebras stabilized by the compact operators.
\end{rem}

 \begin{thm} \label{MainThm}
 The restricted functor $\PrK:\Csep\functor\DGcorr$ factors through $\KKcat$; in other words, we have the following commutative diagram of functors: 
 \beqn
 \xymatrix{ 
\Csep
\ar[rr]^{\PrK}
\ar[dr]_{\iota}
&& \DGcorr\; .\\
& \KKcat
\ar@{-->}[ur]_{\PrK}}
\eeqn
 \end{thm}
 
 \begin{proof}
 We have already checked that the functor $\PrK$ is $C^*$-stable (Lemma \ref{KStable}) and split exact (Lemma \ref{splitExact}). It remains to apply Higson's characterization of $\iota:\Csep\functor\KKcat$ as the universal $C^*$-stable and split exact functor on $\Csep$, see \cite{Hig1,Hig2}.
 \end{proof}

\noindent
Since the author could not find an analogue of Higson's universal characterization of $\KK$-theory on $\CAlg$ in the literature, the above Theorem is stated only on the full subcategory $\Csep$. In fact, in Section \ref{Corr} the objects of the category $\KKcat$ were by definition only separable $C^*$-algebras.

\begin{rem}
We have the following sequence of implications for  separable $C^*$-algebras

\beqn
\text{Morita--Rieffel equiv.}\Rightarrow \text{$\KKcat$ equiv.}\Rightarrow \text{$\PrK$-equiv. in $\DGcorr$}\Rightarrow\text{isom. $\K^{\textup{top}}$-theories}. 
\eeqn In order to make the constructions a bit more sensitive to the underlying topologies of the $C^*$-algebras, one might consider replacing the category of $\CC$-linear spaces by that of locally convex topological vector spaces over $\CC$, which admits the structure of a {\it quasiabelian} category \cite{Schn,Pros}. 
 \end{rem} 

\noindent
It is useful to know that the functor $\PrK:\KKcat\functor\DGcorr$ exists by abstract reasoning. However, in order to make the situation a bit more transparent we make use of a rather algebraic formulation of $\KK$-theory \cite{CunKK}. For any  $C^*$-algebra $A$ let $A\ast A$ denote the free product (which is the coproduct in $\Csep$) of two copies of $A$ and let $qA$ be the kernel of the fold map $A\ast A\map A$. It was shown in {\it ibid.} that $A$ (resp. $B$) is isomorphic to $qA$ (resp. $B_\cpt$) in $\KKcat$ and $\KK_0(A,B)\cong[qA,B_\cpt]$, i.e., homotopy classes of $*$-homomorphisms $qA\map B_\cpt$. Roughly, the algebra $qA$ is expected to play the role of a cofibrant replacement of $A$ and $B_\cpt$ that of a fibrant replacement of $B$ with respect to some model structure with $\KK_*$-equivalences as weak equivalences. This goal has been accomplished in a larger category \cite{JoaJoh}, where all objects are fibrant and a minor modification of $qA$ acts as a cofibrant replacement \cite{JoaJoh}. One can view Kasparov's product simply as a composition of $*$-homomorphisms via the identification $\KK_0(A,B)\cong[q(A_\cpt)\prot\cpt,q(B_\cpt)\prot\cpt]$. In order to define the abelian group structure one proceeds roughly as follows: for any $\phi, \psi\in [qA,B_\cpt]$ one defines $\phi\oplus\psi : qA \map \mathbb{M}_2(B_\cpt)$ as 
 $\left( \begin{smallmatrix}
 \phi & 0 \\
 0    &  \psi
 \end{smallmatrix}
 \right)$. Then one argues that $\mathbb{M}_2(B_\cpt)$ is isomorphic to $B_\cpt$ in $\KKcat$ and fixing such an isomorphism $\theta$ one sets $\phi + \psi = \theta(\phi\oplus\psi)$. We can exploit this fact by concluding that if $A$ is isomorphic to $B$ in $\KKcat$ then $\PrK(qA)$ is isomorphic to $\PrK(B_\cpt)$ in $\DGcorr$ by an explicit DG functor induced by the homotopy class of a $*$-homomorphism $qA\map B_\cpt$ that is invertible in $\KKcat$. That the functor $\PrK:\KKcat\functor\DGcorr$ is additive follows from the universal property. It can be made more explicit. We simply use the fact that any direct sum diagram in $\KKcat$ can be expressed as a split exact diagram involving only $*$-homomorphisms applying $q(-)$ and ${-}\prot\cpt$ several times, both of which produce isomorphic objects in $\KKcat$, and  then apply the split exactness of $\PrK$. In the classification programme of $C^*$-algebras the Kirchberg--Phillips Theorem states that two stable {\it Kirchberg algebras} $A, B$ are isomorphic in $\KKcat$ if and only if they are $*$-isomorphic \cite{KirPhi,NCPClassification}. Therefore, for such algebras $A,B$ the isomorphisms in $\DGcorr$ between $\PrK(A),\PrK(B)$ are more tractable. The automorphism groups of commutative $C^*$-algebras in $\KKcat$ are computable from their $\K^{\textup{top}}$-theory groups using the Universal Coefficient Theorem \cite{RosSch}. 

 \begin{prop}
 For any $A\in\Csep$, the group $\Aut_{\KKcat}(A)$ acts on $\PrK(A)$ in $\DGcorr$ and subsequently on $\bK^{\textup{top}}(A_\cpt)\langle 0\rangle$ in $\HoSpt$ via the functor $\bK(-)$. 
\end{prop} 

\noindent
It would be interesting if, under some assumptions, one could lift the action on $\bK^{\textup{top}}(A_\cpt)\langle 0\rangle$ to the actual category of spectra (and not its homotopy category $\HoSpt$).

 \begin{ex} \label{TorusEx}
Let $A= \C(E)$ be the $C^*$-algebra of continuous functions on a complex elliptic curve $E$. Topologically $E$ is isomorphic to a $2$-torus $\mathbb{T}^2$. It is known that $\K^\textup{top}_0(A)$ is isomorphic to $\ZZ^2$ and so is $\K_1^\textup{top}(A)$. Using the fact that $A$ belongs to the Universal Coefficient Theorem class, one computes $\KK_0(A,A)\simeq \textup{M}(2,\ZZ)\times\textup{M}(2,\ZZ)$. 
 
Let $\textup{D}^b(E)$ be the bounded derived category of coherent sheaves on $E$. The automorphism group $\Aut(\textup{D}^b(E))$ can be described explicitly (see, e.g., Remark 5.13. (iv) \cite{BurKre}). There is a surjective map $\Aut(\textup{D}^b(E))\map\SL(2,\ZZ)$ with a non-canonical splitting defined by sending the generators of $\SL(2,\ZZ)$ to some specific Seidel--Thomas twist functors \cite{SeiTho}. It is clear that the group $\Aut(\textup{D}^b(E))$ is bigger than the units in the ring $\KK_0(A,A)$, since it contains $\textup{Pic}^0(E)$ as a subgroup. The $\SL(2,\ZZ)$ part of $\Aut(\textup{D}^b(E))$ can be described by the Seidel--Thomas twist functors, which can also be seen as Fourier--Mukai transforms and it seems that $\KKcat$-equivalences give rise to analogous actions on $\PrK(A)$. 
\end{ex}

\section{Towards noncommutative topological field theories} \label{TFT}

Following Atiyah--Segal one roughly defines an $n$-dimensional topological field theory to be a symmetric monoidal functor from the bordism category of oriented $(n-1)$-manifolds to some symmetric monoidal category with `sufficiently linear' objects \cite{AtiTQFT}. In string theory there are various dualities, e.g., $\TT$-duality, $\mathbb{S}$-duality, etc., which can also be studied from the point of view of topological field theories. Quite often the construction of dual theories require us to go beyond the realm of commutative geometry (see, e.g., \cite{MatRos3}). Therefore, it is desirable to develop a notion of a noncommutative topological field theory which can accept noncommutative spaces as input. Topological $\K$-theory carries important information about the field theory; for instance, it classifies the D-brane charges of the space \cite{Wit}. The bordisms signify the time evolution of spaces and hence they are expected to induce natural morphisms between topological $\K$-theories according to the way the $D$-brane charges change as one space evolves to another. Since Kasparov's bivariant $\K$-theory is supposed to be the space of morphisms between topological $\K$-theories, we expect every bordism to produce an element in it. More precisely, one would like to have a category of noncommutative bordisms admitting a natural functor to the category $\KKcat$. However, a topological field theory is not expected to factor through $\KKcat$, unless the target category is suitably localized. The construction of the noncommutative bordism category will be taken up later. In this article we produced a functorial construction from $\KKcat$ to a localized and linearized version of the category of essentially small differential graded (DG) categories. It will be useful to detect changes in topological $\K$-theory via bordisms, when noncommutative topological field theories are eventually constructed. One can also construct cyclic homology of DG categories and related Chern character type maps.

\noindent
We give two instances below, where such a formalism can be helpful.

\begin{ex}(Topological $\TT$-dualities) \label{T-duality}

\noindent
A {\it sigma model} roughly studies maps $\Sigma\map X$, where $\Sigma$ is called the {\it worldsheet} (Riemann surface) and $X$ the {\it target spacetime} (typically a $10$-dimensional manifold in supersymmetric string theories). Mirror symmetry relates the sigma models of type $\mathtt{IIA}$ and $\mathtt{IIB}$ string theories with dual Calabi--Yau target spacetimes. In open string theories, i.e., when $\Sigma$ has boundaries, the boundaries are constrained to live in some special submanifolds of the spacetime $X$. Such a submanifold also comes equipped with a special {\it Chan-Paton} vector bundle and together they define a topological $\K$-theory class of $X$ via the Gysin map. The D-brane charges are classified by the topological twisted $\K$-theory classes, in the presence of $\H$-fluxes. The homological mirror symmetry conjecture of Kontsevich predicts an equivalence of triangulated categories of $\mathtt{IIA}$-branes (Fukaya category) on a Calabi--Yau target manifold $X$ and $\mathtt{IIB}$-branes (derived category of coherent sheaves) on its dual $\hat{X}$. This equivalence would induce an isomorphism between their Grothendieck groups and it was argued that the Grothedieck group of the category of $A$-branes on $X$, at least when the dimension of $X$ is not divisible by $4$, should be isomorphic to $\K^1_{\textup{top}}(X)$ \cite{Horava}. Strominger--Yau--Zaslow argued that sometimes when $X$ and $\hat{X}$ are mirror dual Calabi--Yau $3$-folds one should be able to find a generically $\TT^3$-fibration over a common base $Z$

\begin{equation} \label{fibration}
\xymatrix{
X
\ar@{-->}[rr]^{\text{$\TT$-duality}}
\ar[dr]
&& \hat{X} \, ,
\ar[dl] \\
& Z
}
\end{equation}

such that mirror symmetry is obtained by applying $\TT$-duality fibrewise \cite{SYZ}. Since $\TT$-duality is applied an odd number of times it interchanges types ($\mathtt{IIA}\leftrightarrow \mathtt{IIB}$). Sometimes using Poincar{\' e} duality type arguments it is possible to identify topological $\K$-theory with $\K$-homology. Kasparov's $\KK$-theory naturally subsumes $\K$-theory and $\K$-homology. It was shown in \cite{BMRS2} that the effect of certain odd number of topological $\TT$-duality transformations (possibly including more parameters like $\H$-fluxes) could be interpreted as arising from $\KK_1$-classes between suitably defined continuous trace stable $C^*$-algebras, which capture the geometry of the above diagram \ref{fibration}. As argued before, whilst the application of an odd number of $\TT$-duality transformations interchanges types ($\mathtt{IIA}\leftrightarrow\mathtt{IIB}$), that of an even number preserves types and corresponds to a $\KK_0$-class. Therefore, an invertible topological $C^*$-correspondence or an invertible $\KK_0$-class is an abstract generalization of an even number of $\TT$-duality transformations (or a $\TT^{2n}$-duality), viewed as an equivalence of $\mathtt{IIB}$-branes (or $\mathtt{IIA}$-branes) and inducing an isomorphism at the level of topological $\K$-theories. The category $\KKcat$ is triangulated with the suspension functor $\Sigma(-)=\C_0((0,1))\prot -)$ and one can also use the identification $\KK_1(A,A')\cong \KK_0(A,\Sigma A')$ to reduce a $\TT^{2n+1}$-duality to a $\TT^{2n}$-duality.
\end{ex}

\begin{ex}($\E$-theory and strong deformations up to homotopy) \label{deformation}

\noindent
The Connes--Higson $\E$-theory \cite{ConHig} is the universal $C^*$-stable, half exact and homotopy invariant functor. Recall that a functor $F$ is called {\it half exact} if it sends an exact sequence of $\C^*$-algebras $0\map A\map B\map C\map 0$ to an exact sequence $F(A)\map F(B)\map F(C)$ (only exact at $F(B)$). It is known that any half exact and homotopy invariant functor is split-exact. Therefore, there is a canonical induced functor $\KKcat\functor\mathtt{E}$, where the category $\mathtt{E}$ consists of $C^*$-algebras with bivariant $\E_0$-groups as morphisms. This functor is fully faithful when restricted to nuclear $C^*$-algebras; in fact, by the Choi--Effros lifting Theorem $\KK_*(A,B)\cong\E_*(A,B)$ whenever $A$ is nuclear. Thus, restricted to nuclear $C^*$-algebras there are maps $\E_0(A,B)\cong\KK_0(A,B)\map \Hom_{\DGcorr}(\PrK(A),\PrK(B))$. Let $B_\infty:= \C_b([1,\infty),B)/\C_0([1,\infty),B)$, where $\C_b$ denotes bounded continuous functions, be the {\it asymptotic algebra} of $B$. An {\it asymptotic morphism} between $C^*$-algebras $A$ and $B$ is a $*$-homomorphisms $\phi:A\map B_\infty$.
 
A {\it strong deformation} of $C^*$-algebras from $A$ to $B$ is a continuous field $A(t)$ of $C^*$-algebras over $[0,1]$ whose fibre at $0$, $A(0)\cong A$ and whose restriction to $(0,1]$ is the constant field with fibre $A(t)\cong B$ for all $t\in (0,1]$. 

Given any such strong deformation of $C^*$-algebras and an $a\in A$ one can choose a section $\alpha_a(t)$ of the continuous field such that $\alpha_a(0) =a$. Suppose one has chosen such a section $\alpha_a(t)$ for every $a\in A$. Then one associates an asymptotic morphism by setting $(\phi(a))(t) = \alpha_a(1/t), \; t\in[1,\infty)$. Recall that $\Sigma A:=\C_0((0,1),A)$ is the suspension of $A$. The Connes--Higson picture of $\E$-theory says that $\E_0(A,B) \cong [[\Sigma A_\cpt,\Sigma B_\cpt]]$, where $[[?,-]]$ denotes homotopy classes of asymptotic morphisms between $?$ and $-$. Let $\phi$ be an asymptotic morphism defined by a strong deformation from $A$ to $B$. Then we call the class of $\phi$ in $\E_0(A,B)$ a {\it strong deformation up to homotopy} from $A$ to $B$. If $A$ is nuclear, e.g., if $A$ is commutative, whenever the class of $\phi$ is invertible in $\E_0(A,B)$, we deduce that $\PrK(A)$ and $\PrK(B)$ are isomorphic in $\DGcorr$.
\end{ex}


\bibliographystyle{abbrv}
\bibliography{/home/ibatu/Professional/math/MasterBib/bibliography}

\medskip

\end{document}